\documentclass{amsart}

\usepackage{graphicx}
\usepackage[margin=1.2in]{geometry}
\usepackage{tikz}
\usepackage{tikz-cd}
\usepackage{comment}
\usepackage{mathrsfs}
\usepackage{subfig}

\usepackage{amsmath}
\usepackage{amsthm}
\usepackage{amsfonts}
\usepackage{amssymb}
\usepackage{graphicx}
\graphicspath{ {images/} }
\usepackage{pgfplots}
\usepackage{tikz}
\usepackage{url}
\usepackage{mathrsfs}
\usepackage{comment}
\usepackage[shortlabels]{enumitem}
\usepackage{tikz-cd}
\usetikzlibrary{arrows}

\newtheorem{theorem}{Theorem}[section]
\newtheorem{proposition}[theorem]{Proposition}
\newtheorem{lemma}[theorem]{Lemma}
\newtheorem{corollary}[theorem]{Corollary}

\theoremstyle{definition}
\newtheorem{definition}[theorem]{Definition}

\theoremstyle{remark}
\newtheorem{remark}[theorem]{Remark}

\numberwithin{equation}{section}

\newcommand{\set}[2]{\left\{\,#1~:~#2\,\right\}}

\DeclareMathOperator{\ind}{Ind}

\newcommand{\R}{\mathbb{R}}
\newcommand{\Lie}{\mathcal{L}}

\newcommand{\ve}{\varepsilon}

\pgfplotsset{compat=1.12}

\DeclareMathOperator{\diver}{div}

\usepackage{hyperref}
\hypersetup{
    colorlinks=true,
    linkcolor=blue,
}

\begin{document}

\title{Morse-Smale characteristic foliations and convexity in contact manifolds}
\author{Joseph Breen}

\maketitle

\begin{abstract}
We generalize a result of Giroux which says that a closed surface in a contact $3$-manifold with Morse-Smale characteristic foliation is convex. Specifically, we show that the result holds in contact manifolds of arbitrary dimension. As an application, we show that a particular closed hypersurface introduced by A. Mori is $C^{\infty}$-close to a convex hypersurface. 
\end{abstract}

\tableofcontents

\section{Introduction}

In \cite{Giroux}, Giroux demonstrated the power of convex surface theory in three dimensional contact manifolds. Since then, convexity has been an effective tool in this setting; see for example \cite{Honda1}. Recently, a systematic development of convex hypersurface theory in arbitrary dimensions began in works such as \cite{HondaHuang2}, \cite{HondaHuang}, and \cite{Sackel}. The goal of this paper is to study further one aspect of convexity in higher dimensions.

In particular, one of Giroux's results in \cite{Giroux} is that a closed surface in a $3$ dimensional contact manifold with Morse-Smale characteristic foliation is convex. We recall the relevant definition. 

\begin{definition}\label{MSdef}
A vector field on an oriented manifold is \textbf{Morse-Smale} if the following conditions are satisfied: 
\begin{enumerate}[(i)]
    \item There are finitely many critical points and periodic orbits, each of which is hyperbolic (in the dynamical systems sense). 
    \item Every flow line limits to either a critical point or an orbit in both forward and backward time.
    \item The unstable manifold of any critical point or orbit is transverse to the stable manifold of any critical point or orbit. 
\end{enumerate}
A singular foliation is \textbf{Morse-Smale} if it is directed by a Morse-Smale vector field.
\end{definition}

In \cite{HondaHuang}, Honda and Huang adapted Giroux's argument to show that a hypersurface in a contact manifold of arbitrary dimension with so called Morse$^+$ characteristic foliation is convex. The Morse$^+$ hypothesis, which requires the existence of a Morse function for which the foliation is gradient-like, precludes the existence of periodic orbits in the characteristic foliation. Here, we generalize further to include the case where the foliation has periodic orbits. Our main result is the following.

\begin{theorem}\label{maintheorem}
Let $\Sigma^{2n} \subseteq (M^{2n+1},\xi= \ker \alpha)$ be a closed, oriented hypersurface with Morse-Smale characteristic foliation. Then $\Sigma$ is convex. 
\end{theorem}

\begin{remark}
The $+$ in the Morse$^+$ hypothesis in \cite{HondaHuang} is the assumption that there are no trajectories from negative singularities to positive singularities. It will be evident from the proof of Theorem \ref{maintheorem} that the analogue of this assumption in Definition \ref{MSdef} is condition (iii). Also worth nothing is that Honda and Huang prove that a hypersurface with Morse characteristic foliation can be smoothly perturbed to have Morse$^+$ characteristic foliation.
\end{remark}

\begin{remark}
When $\dim M = 3$, Theorem \ref{maintheorem} (i.e., Giroux's original result) is especially powerful because Morse-Smale vector fields on $2$-manifolds are dense in the $C^{\infty}$-topology (see \cite{Palis} and the references within). This implies that a $C^{\infty}$-generic closed surface has Morse-Smale characteristic foliation, and thus is convex. Morse-Smale vector fields are \it not \rm $C^{\infty}$-dense in higher dimensions. 
\end{remark}

The proof of Theorem \ref{maintheorem} relies on an understanding of the induced $1$-form $\beta = \alpha\mid_{\Sigma}$ of a contact form $\alpha$ near periodic orbits. The terminology we will use in this paper is:

\begin{definition}\label{liouville_orbit}
Let $\beta := \alpha\mid_{\Sigma}$. A periodic orbit $\gamma$ in the characteristic foliation $\Sigma_{\xi}$ is \textbf{Liouville} if $g\beta$ is a Liouville form in a neighborhood of $\gamma$ for some smooth $g>0$. We say $\gamma$ is \textbf{positive Liouville} if $d(g\beta)^n > 0$ and \textbf{negative Liouville} if $d(g\beta)^n < 0$.
\end{definition}

\begin{remark}\label{divergence}
Here is a simple criterion for an orbit to be Liouville: pick any volume form $\Omega$ in a neighborhood of $\gamma$ and consider the vector field $X$ satisfying $i_X\Omega = \beta\, (d\beta)^{n-1}$ which directs the characteristic foliation. If $\diver_{\Omega}X \neq 0$, then $\gamma$ is Liouville. Indeed, 
$$
\diver_{\Omega}(X)\, \Omega = d\left(i_{X}\Omega\right) = d\left(\beta \, (d\beta)^{n-1}\right) = (d\beta)^n
$$
so that $d\beta$ is symplectic if $\diver_{\Omega}(X) \neq 0$. One may easily check that the sign of $\diver_{\Omega}(X)$ is independent of the choice of $\Omega$. 
\end{remark}

The proof that Morse$^+$ implies convexity relies on the fact that $\beta$ is a Liouville form in a neighborhood of a critical point of the characteristic foliation. Also important is the fact that the Morse index of a critical point of a Liouville vector field satisfies $\ind(p) \leq n$, where $2n$ is the dimension of the Liovuille manifold (see Proposition 11.9 of \cite{CE}). One of the main steps in proving Theorem \ref{maintheorem} is to show that hyperbolic periodic orbits exhibit the same behavior. 

\begin{proposition}\label{orbit_prop}
Let $\Sigma^{2n}\subseteq (M^{2n+1}, \xi = \ker \alpha)$ be an oriented hypersurface. If $\gamma$ is a hyperbolic periodic orbit in the characteristic foliation, it is Liouville. Furthermore, if $\gamma$ is positive Liouville then $\ind(\gamma) \leq n$. 
\end{proposition}

With this and a few other ingredients, the proof of Theorem \ref{maintheorem} is a straightforward adaptation of Giroux's argument in three dimensions; see also the proof of Proposition 2.2.3 in \cite{HondaHuang}.

As an application of this convexity criterion, we provide some further analysis on a closed hypersurface $\Sigma_0$ introduced by Mori in \cite{Mori}. We will review the definition of $\Sigma_0$ in Section \ref{Applications}. In \cite{Mori} it was claimed that $\Sigma_0$ cannot be smoothly approximated by a convex hypersurface. Using Theorem \ref{maintheorem}, we will prove:

\begin{corollary}\label{maincorollary}
The closed hypersurface $\Sigma_0$ is $C^{\infty}$-close to a convex hypersurface. 
\end{corollary}

\begin{remark}
We emphasize that our work only shows that the \textit{closed} hypersurface $\Sigma_0$ can be smoothly approximated by a convex hypersurface. In \cite{Mori}, Mori also introduces a hypersurface with contact type boundary and states a conjectural Thurston-Bennequin-like inequality for convex hypersurfaces; see also \cite{Mori2}. Theorem \ref{maintheorem} and the proof of Corollary \ref{maincorollary} do not apply to the hypersurface with boundary, or disprove the conjectured inequality. 
\end{remark}

This paper is organized as follows. Section \ref{background} contains the necessary background material on characteristic foliations and convexity in contact manifolds, as well as some notions from dynamical systems. In Section \ref{proof}, Theorem \ref{maintheorem} is proved. Specifically, we prove Proposition \ref{orbit_prop} and use this to prove Theorem \ref{maintheorem}. Section \ref{Applications} contains the analysis of Mori's example.

\vspace{2mm}
\noindent \textbf{Acknowledgements}. The author would like to thank Ko Honda for numerous helpful ideas and patient suggestions, as well as Atsuhide Mori for an insightful correspondence.

\section{Background material}\label{background}

We assume familiarity with basic contact and symplectic geometry; we relegate further details to \cite{Geiges}. In this paper, all of our contact manifolds are oriented and our contact structures are co-oriented.

\begin{definition}
If $\Sigma$ is a hypersurface in a contact manifold $(M, \xi = \ker \alpha)$, the \textbf{characteristic foliation} is the singular $1$-dimensional foliation 
$$
\Sigma_{\xi} = \left(T\Sigma \cap \xi\mid_{\Sigma}\right)^{\bot}
$$
where $\bot$ is the symplectic orthogonal complement taken with respect to the conformal symplectic structure on $\xi$. If $\beta := \alpha\mid_{\Sigma}$, then 
$$
\Sigma_{\xi} = \ker \left(d\beta \mid_{\ker \beta}\right). 
$$
\end{definition}

If $\Sigma$ is oriented, $\Sigma_{\xi}$ inherits a nautral orientation. In this case, a convenient way to compute the characteristic foliation on an orientable hypersurface $\Sigma^{2n}\subset M^{2n+1}$ is given by Lemma 2.5.20 in \cite{Geiges}.

\begin{lemma}\cite{Geiges}\label{cflemma}
Let $\beta = \alpha\mid_{\Sigma}$ and let $\Omega$ be a volume form on $\Sigma$. The characteristic foliation $\Sigma_{\xi}$ is directed by the vector field $X$ satisfying 
\begin{equation}
i_X \Omega = \beta\, (d\beta)^{n-1}.    
\end{equation}
\end{lemma}

In three dimensional contact manifolds, the characteristic foliation alone determines the contact germ near a hypersurface \cite{Giroux}. In higher dimensions we have the following weaker fact. 

\begin{lemma}\cite{HondaHuang}\label{conformallemma}
Let $(M, \xi_i = \ker \alpha_i)$ for $i=0,1$ be two contact structures on the same manifold. Let $\beta_i = \alpha_i\mid_{\Sigma}$ and suppose that $\beta_0 = g\beta_1$ for some $g > 0$. Then there is an isotopy $\phi_s:M\to M$ such that $\phi_s(\Sigma) = \Sigma$, $\phi_0 = \text{id}_M$, and $(\phi_1)_*(\xi_0) = \xi_1$ in a neighborhood of $\Sigma$. 
\end{lemma}

Any submanifold of $\Sigma$ transverse to the characteristic foliation is a contact submanifold of $M$. Furthermore, flowing along the characteristic foliation induces a contactomorphism of the transversal.

\begin{definition}
A \textbf{contact vector field} $V$ in a contact manifold $(M,\xi= \ker \alpha)$ is one whose flow $\phi_t:M\to M$ is a contactomorphism for all $t$. 
\end{definition}

There is a one-to-one correspondence between contact vector fields $V$ and ``contact Hamiltonian functions'' $C^{\infty}(M)$, see Section 2.3 of \cite{Geiges}. Given $H\in C^{\infty}(M)$, the corresponding contact vector field is determined uniquely by the conditions 
\begin{equation}\label{cvf}
\alpha(X_H) = H \quad \quad \text{ and } \quad \quad i_{X_H}d\alpha = dH(R_{\alpha})\, \alpha - dH.    
\end{equation}
A vector field $V$ is contact if and only if $\Lie_V\alpha = g\alpha$ for some smooth $g:M\to \R$. The Reeb vector field $R_{\alpha}$ is an example of a contact vector field. 

\begin{definition}
A hypersurface $\Sigma\subset (M,\xi)$ is \textbf{convex} if there is a contact vector field $V$ everywhere transverse to $\Sigma$. 
\end{definition}  

One can characterize convexity at the differential form level as follows. 

\begin{lemma}\cite{Giroux}
An embedded oriented  hypersurface $\Sigma$ is convex if and only if there is an neighborhood $\Sigma \times \R$ of $\Sigma=\Sigma\times \{0\}$ in $M$ such that $\xi = \ker(u\, dt + \beta)$, where $t$ is the $\R$-coordinate, $\beta$ is a ($t$-independent) $1$-form on $\Sigma$, and $u$ is a ($t$-independnet) function $u:\Sigma \to \R$. 
\end{lemma}

Note that \it any \rm $1$-form on $\R\times \Sigma$ can be written $u_t\, dt + \beta_t$ for some family of smooth functions $u_t:\Sigma \to \R$ and family of $1$-forms $\beta_t$ on $\Sigma$. Convexity requires a form which is $t$-invariant. A convex hypersurface is naturally divided into three regions in the following way. Write $\alpha = u\, dt + \beta$ near $\Sigma$. Then
$$
R^+(\Sigma) = \{u > 0\} \quad \text{ and }\quad  R^-(\Sigma) = \{u < 0\} 
$$
are the positive and negative region, respectively, and $\Gamma = \{u = 0\}$ is a codimension 1 submanifold of $\Sigma$ called the dividing set. The dividing set (which depends on the choice of contact vector field) is well-defined up to isotopy of dividing sets.

Next, we recall one definition from symplectic geometry. 

\begin{definition}
A \textbf{Liouville form} on a symplectic manifold $(W,\omega)$ is a $1$-form $\beta$ such that $\omega = d\beta$. The vector field $X$ such that $i_X\omega = \beta$ is the \textbf{Liouville vector field} of $\beta$. 
\end{definition}

In a convex hypersurface, $R^+(\Sigma)$ and $R^-(\Sigma)$ inherit a Liouville structure from $\beta = \alpha\mid_{R^{\pm}(\Sigma)}$. If $X$ denotes the Liouville vector field (for either $R^+(\Sigma)$ or $R^-(\Sigma)$), then the characteristic foliation on $R^+(\Sigma)$ is directed by $X$ and the characteristic foliation on $R^-(\Sigma)$ is directed by $-X$.

Finally, the dynamical systems notion of hyperbolicity will be central in what follows. We refer to \cite{Palis} for more details.

\begin{definition}
Let $\gamma$ be a periodic orbit of a vector field $X$, and let $L$ be a transversal to $X$ which intersects $\gamma$ once. The \textbf{Poincare first return map} is the map $P:V\subset L\to L$ defined by following the trajectories of $X$ from some open subset $V$ of $L$ to their first point of return to $L$. The orbit $\gamma$ is \textbf{hyperbolic} if the eigenvalues $\mu$ of $TP$ satisfy $0< |\mu| \neq 1$.
\end{definition}

\section{Proof of Theorem \ref{maintheorem}}\label{proof}

As alluded to in the introduction, we begin by proving Proposition \ref{orbit_prop}, which allows us to definitively place an orbit in either the positive or negative region.

\begin{proof}[Proof of Proposition \ref{orbit_prop}.]
The general strategy of the proof is to show that the divergence of a vector field directing the characteristic foliation near the hyperbolic periodic orbit $\gamma$ is nonzero. By Remark \ref{divergence}, this proves that $\gamma$ is Liouville.

\vspace{2mm}
\indent \underline{Step 1: Analyzing the differential of the Poincare first-return map}.
\vspace{1mm}

Let $L$ be a transversal to the periodic orbit and $V\subset L$ an open subset diffeomorphic to $\R^{2n-1}$ containing $\{0\} = \gamma \cap L$ such that the Poincare first-return map $P:V\to L$ is defined. Let $\lambda = \alpha\mid_L$ be the induced contact form on $L$. Because $P$ is defined by following the trajectories of the flowlines of $\Sigma_{\xi}$, $P$ is a contactomorphism. Thus, $P^*\lambda = f\lambda$ for some $f > 0$. 

Next, we compute a matrix representative for $T_0P:T_0V\to T_0L$. Let $R = R_{\lambda}(0)$, the Reeb vector field for $\lambda$ at $0$, and let $\{R, v_1, \dots, v_{2n-2}\}$ be a basis for $T_0V$ such that $\{v_1, \dots, v_{2n-2}\}$ is a symplectic basis for $\ker \lambda$ with respect to the symplectic structure induced by $d\lambda$. Write $T_0P(R) = C\, R + v$ for some constant $C$ and some $v\in \ker \lambda$. Since $P$ is a contactomorphism, $\ker \lambda$ is invariant under $P$. Thus, with respect to the above basis we have 
$$
[T_0P] = \begin{pmatrix}
C & 0 \\
\ast & M
\end{pmatrix}
$$
where $0$ is a $1\times (2n-2)$ matrix of zeroes and $\ast$ is a $(2n-2) \times 1$ matrix determined by $T_0P(v)$. Since 
$$
C = \lambda(C\, R + v) = \lambda(T_0P(R)) = P^*\lambda(R) = f(0)
$$
it follows that one of the eigenvalues of $T_0P$ is $C = f(0)$. The assumption that $\gamma$ is hyperbolic is precisely the assumption that the eigenvalues of $T_0P:T_0V\to T_0L$ satisfy $|\mu|\neq 1$ (and $\mu\neq 0$). Thus, $C>1$ or $0<C<1$.

Next, since $P^*d\lambda = df\, \lambda + f\, d\lambda$, 
$$
P^*d\lambda\mid_{\ker\lambda} = f(0)\, d\lambda\mid_{\ker\lambda}.
$$
This implies that $M^TJ_0M = C\, J_0$, where $J_0$ is the skew-symmetric matrix corresponding to the symplectic structure on $\ker \lambda$. Let $M' = C^{-\frac{1}{2}}M$. Then 
$$
(M')^TJ_0(M') = C^{-1} M^TJ_0M = J_0
$$
so that $M'$ is a symplectic matrix. Thus, to summarize Step 1: 
\begin{equation}\label{poincare_matrix}
[T_0P] = \begin{pmatrix}
C & 0 \\
\ast & \sqrt{C}\, M'
\end{pmatrix}    
\end{equation}
where either $0<C<1$ or $C>1$, and $M'$ is a symplectic matrix.

\vspace{2mm}
\indent \underline{Step 2: Determining the divergence of the characteristic foliation}.
\vspace{1mm}

Let $X$ be a vector field directing the characteristic foliation near $\gamma$. Let $\theta$ be a coordinate on $\gamma$. By considering a volume form $\Omega = d\theta\, \Omega'$ where $\Omega$ is a (possibly $\theta$-dependent) volume form in the transverse direction, we may assume that $X = \partial_{\theta} + Y$ where $Y$ is a vector field in the transverse direction which has a hyperbolic zero at $0$. By Remark \ref{divergence}, to show that $\gamma$ is Liouville it suffices to show that $\diver(Y) (= \diver(X))$ is nonzero along $\gamma$.

Reparametrizing if necessary, we may further assume that $P(x) = \phi_t(x)$, where $\phi_1$ is the flow of $Y$. By the Hartman-Grobman theorem (see Section 2.4 of \cite{Palis}), in a small neighborhood of $0$ it is sufficient to consider the flow of the linearization of $Y$, which we denote by $Ax$. Here $x\in \R^{2n-1}$ and $A$ is a square matrix. 

Note that $(\diver Y)(0) = \text{tr} A(0)$. Because $P(x) = \phi_1(x)$, by standard linear dynamical systems theory it follows that $[T_0P] = e^{A(0)}$. Since $\det(e^A) = e^{\text{tr} A}$, 
$$
(\diver Y)(0) = \text{tr} A(0) = \log \det [T_0P].
$$
Since the determinant of any symplectic matrix is $1$, \eqref{poincare_matrix} implies $\det [T_0P] = C\cdot \sqrt{C}^{2n-2} = C^n$. Thus, 
$$
(\diver Y)(0) = n\log C.
$$
Since $C>1$ or $0<C<1$, this shows that $\diver X\neq 0$ in a sufficiently small neighborhood of $\gamma$. 

This proves that a hyperbolic orbit is Liovuille. In particular, if $f(0) > 1$ then $\gamma$ is positive Liouville and if $f(0) < 1$ then $\gamma$ is negative Liouville.

\vspace{2mm}
\indent \underline{Step 3: Computing the index of a positive orbit}.
\vspace{1mm}

Suppose that $\gamma$ is a positive hyperbolic orbit. Consider $[T_0P]$ as in \eqref{poincare_matrix}. Since $\gamma$ is positive, $C > 1$. Let $E^-$ denote the subspace of generalized eigenvectors with eigenvalues of modulus $<1$. We claim that $\dim E^- \leq n-1$. The final claim in the proposition then follows, as the dimension of the stable manifold is $\leq (n-1) + 1 = n$ after accounting for the orbit direction. 

It is a standard fact (see, for example, \cite{MS}) that if $\mu'$ is an eigenvalue of a symplectic matrix $M'$, then $(\mu')^{-1}$ is also an eigenvalue with the same multiplicity. This implies that if $\mu$ is an eigenvalue of $M = \sqrt{C}\, M'$, then $C\mu^{-1}$ is also an eigenvalue of $M$ with equal multiplicity. In particular, if $|\mu|<1$ then $|C\mu^{-1}| > 1$. Thus, there are at most $n-1$ eigenvalues of $M$ with modulus less than $1$, which proves the claim. 

\end{proof}

Now we can adapt the arguments in \cite{Giroux} and \cite{HondaHuang} to prove that a Morse-Smale characteristic foliation is sufficient for convexity in arbitrary dimensions.

\begin{proof}[Proof of Theorem \ref{maintheorem}.]
Suppose that $\Sigma_{\xi}$ is Morse-Smale. To show that $\Sigma$ is convex (up to a contact isotopy of $M$ which fixes $\Sigma$), it suffices by Lemma \ref{conformallemma} to construct a vertically invariant contact form $\alpha_1$ on a neighborhood of $\Sigma$ such that $\beta_1 = g\beta$ for some $g > 0$. Here $\beta_1 = \alpha_1\mid_{\Sigma}$ and $\beta = \alpha\mid_{\Sigma}$. Throughout the proof we will loosely use $g$ to denote a sufficient positive function.

Classify each singular point $p$ of $\Sigma_{\xi}$ as either positive or negative in the natural way, i.e., based on the orientations of $\xi_p$ and $T_p\Sigma$. Classify each periodic orbit as either positive or negative according to Proposition \ref{orbit_prop}.

We claim that there is no flow line from a negative critical point or orbit to a positive critical point or orbit. Indeed, as mentioned in the introduction, the Morse index of a positive critical point $p$ satisfies $\ind(p) \leq n$. By Proposition \ref{orbit_prop}, the Morse index of any positive orbit also satisfies $\ind(\gamma)\leq n$. The transversality assumption in Definition \ref{MSdef} implies that the stable manifold of any positive critical point or orbit and the unstable manifold of any negative critical point or orbit either do not intersect, or the dimension of the intersection is $0$. In either case, there can be no flow line (necessarily one-dimensional) from a negative point or orbit to a positive point or orbit. 

Next, we will construct open sets $U^+$ and $U^-$ in $\Sigma$ containing all positive and negative points and orbits, respectively, and then use the resulting decomposition of $\Sigma$ to define $\alpha_1$. In particular, $U^+$ and $U^-$ will be ``prototypes'' for $R^+(\Sigma)$ and $R^-(\Sigma)$.

\vspace{2mm}
\indent \underline{Step 1: Constructing $U^+$}.
\vspace{1mm}

For any set $S \subset \Sigma$, let $\text{Op}(S)$ denote a sufficiently small open neighborhood of $S$ in $\Sigma$. 

Let $\{x_1, x_2, \cdots, x_M\}$ be the list of positive orbits and points (i.e., $x_i$ can be a critical point or an orbit). Let $U_k$ be the union of $\text{Op}(\{x_1, x_2, \cdots, x_k\})$ with sufficiently small tubular neighborhoods of the stable manifolds of $x_1, \dots, x_{k-1}$. Because there is no trajectory from a negative point or orbit, we may assume that the list $\{x_1, \dots, x_M\}$ is ordered so that a tubular neighborhood of the stable manifold of $x_{k+1}$ intersects $\partial U_k$ in a contact submanifold. Here the contact assumption comes from choosing $U_k$ so that $\partial U_k$ is transverse to the characteristic foliation. Finally, let $U^+ = U_M$.

\vspace{2mm}
\indent \underline{Step 2: Defining $\beta_1$ on $U^+$}.
\vspace{1mm}

We will define $\beta_1$ on $U^+$ by inducting on $k$. Note that $U_1 = \text{Op}(x_1)$ and by assumption, $\beta_1 := g\beta$ is positive Liouville on $U_1$ for some $g>0$. Now suppose that $\beta_1$ has been constructed on $U_k$. By assumption, $g\beta$ is Liouville on $\text{Op}(x_{k+1})$. Using the flow of the characteristic foliation and the above remark about the stable manifold of $x_{k+1}$, we may identify 
$$
U_{k+1} \setminus \left(U_k\cup \text{Op}(x_{k+1})\right)
$$
with $[0,1]\times L$ where $L$ is a contact submanifold. Here, $\{0\}\times L \subseteq \partial U_k$ and $\{1\}\times L \subseteq \partial \text{Op}(x_{k+1}).$ Because $L$ is a contact submanifold of $M$, $\lambda = \beta\mid_L$ is a contact form on $L$. Since the flow of the characteristic foliation is a contactomorphism of $L$, we have $\beta = h\, \lambda$ on $[0,1]_s\times L$ for some smooth $h  > 0$. Note that 
$$
d\beta = \frac{\partial h}{\partial s}\, ds\, \lambda + d_Lh\, \lambda + h\, d\lambda 
$$
and so 
\begin{equation}\label{sympeq}
(d\beta)^{n} = nh^{n-1}\, \frac{\partial h}{\partial s}\, ds\, \lambda\, (d\lambda)^{n-1}.    
\end{equation}
Thus, $\beta$ is Liouville if $\frac{\partial h}{\partial s} > 0$. After scaling $\beta$ by a sufficiently large constant on $\text{Op}(x_{k+1})$, the function $h$ can be multiplied by a positive function $\frac{h_1}{h}$ so that $\frac{\partial h_1}{\partial s} > 0$ on $[0,1]_s\times L$. With $\beta_1$ defined on $U_{k+1}$ in this way, $\beta_1$ is positive Liouville on $U_{k+1}$. 

Inductively, this defines $\beta_1$ on $U^+$ so that $\beta_1 = g\beta$ is a positive Liouville form.

\vspace{2mm}
\indent \underline{Step 3: Constructing $U^-$ and defining $\beta_1$ on $U^-$}.
\vspace{1mm}

Define an open neighborhood $U^-$ together with a negative Liouville form $\beta_1 = g\beta$ in the analogous way using negative singular points and negative periodic orbits together with the unstable manifolds of each.

\vspace{2mm}
\indent \underline{Step 4: Defining $\beta_1$ near the dividing set}.
\vspace{1mm}

By the above steps, $U^+$ and $U^-$ are disjoint open sets in $\Sigma$ containing all singular points and orbits. Furthermore, there are no flowlines running from $U^-$ to $U^+$. Thus, using the flow of $\Sigma_{\xi}$ we may identify $\Sigma \setminus (U^+ \cup U^-)$ with $[-1,1]_s\times \Gamma$ for some submanifold $\Gamma$, where $\{-1\}\times \Gamma = \partial U^+$, $\{1\}\times \Gamma = \partial U^-$, and $\Sigma_{\xi}$ is directed by $\partial_s$. Let $\lambda$ be the induced contact form on $\Gamma$. On $[-1,1]\times \Gamma$, $\beta = h\, \lambda$. We then have $\frac{\partial h}{\partial s} > 0$ near $\{-1\}\times \Gamma$ and $\frac{\partial h}{\partial s} < 0$ near $\{1\}\times \Gamma$ (see the remark after \eqref{sympeq}). Multiply $h$ by a function $\frac{h_1}{h}$ so that $\frac{\partial h}{\partial s} > 0$ for $-1\leq s< 0$, $\frac{\partial h_1}{\partial s} = 0$ for $s=0$, and $\frac{\partial h_1}{\partial s} < 0$ for $0<s\leq 1$. Let $\beta_1 = h\, \lambda$.

\vspace{2mm}
\indent \underline{Step 5: Defining the vertically invariant contact form $\alpha_1$ on $\R\times \Sigma$}.
\vspace{1mm}

Decompose $\R_t\times \Sigma$ as 
$$
\left(\R \times U^+\right) \cup \left(\R \times U^-\right) \cup \left(\R \times [-1,1]\times\Gamma\right). 
$$
Let $\alpha_1 = dt + \beta_1$ on $\R \times U^+$ and let $\alpha_1 = -dt + \beta_1$ on $\R \times U^-$. Since $\beta_1$ is positive (negative) Liouville on $U^+$ ($U^-$), $\alpha_1$ defines a contact form on these regions. Furthermore, by construction, $\alpha_1\mid_{U^\pm} = g\beta$ for some $g > 0$. 

To define $\alpha_1$ on $\R \times [-1,1]\times\Gamma$, let $u:[-1,1]\to \R$ be a smooth function such that $u'(s) < 0$ for $-1<s<1$, $u'(s) = 0$ for $|s|=1$, $u(-1) = 1$, $u(0) = 0$, and $u(1) = -1$. Let $\alpha_1 = u\, dt + \beta_1$. Then $\alpha_1$ is a smoothly defined $1$-form on $\R\times \Sigma$, and 
$$
\alpha_1 \, (d\alpha_1)^n = n\left(u^n\,\frac{\partial h_1}{\partial s} - \frac{\partial u}{\partial s}\, h_1^n \right)\, dt\, ds\, \lambda \, (d\lambda)^{n-1}.
$$
One may verify that with an appropriate choice of $h_1$ as defined in Step 4,
\begin{equation}\label{poseq}
u^n\,\frac{\partial h_1}{\partial s} - \frac{\partial u}{\partial s}\, h_1^n > 0.  
\end{equation}
If $n$ is even, then for $0<s\leq 1$ we require 
$$
\left|\frac{\partial h_1}{\partial s}\right| < \left|\frac{\partial u}{\partial s}\right|\cdot \left|\frac{h_1}{u}\right|^n
$$
which can be arranged by making $h_1$ sufficiently flat. Otherwise, the definitions of $u$ and $h_1$ force \eqref{poseq} to hold. Thus, $\alpha_1$ is a vertically invariant contact form defined near $\Sigma$ such that $\alpha_1\mid_{\Sigma} = g\beta$ for some positive function $g$. By the remark at the beginning of the proof, $\Sigma$ is convex.

\end{proof}

\section{Applications}\label{Applications}

In this section we provide some further analysis on a non-convex hypersurface introduced by A. Mori. We begin with some generalities, and then in \ref{morisubsection} we review the definition of the hypersurface, the argument for its non-convexity, and then prove that there is $C^{\infty}$-small perturbation of the hypersurface to a convex hypersurface. 

First, a lemma which computes the perturbation of the characteristic foliation in a particular model. 

\begin{lemma}\label{compute_perturb}
Consider the contact manifold $\R_t \times S^1_{\theta}\times L^{2n-1}$ with contact form $\alpha = t\, d\theta + \lambda$, where $\lambda$ is a contact form on $L$. Let $H:M\to \R$ be a smooth function, and let $\tilde{\Sigma} = \{t = H\}$. Let $X_H$ be the contact vector field corresponding to the contact Hamiltonian $H$ as in \eqref{cvf}. Then the characteristic foliation of $\tilde{\Sigma}$ is directed by $\partial_{\theta} - X_H$.
\end{lemma}

\begin{proof}
Let $\Omega = d\theta\, \lambda\, (d\lambda)^{n-1}$. Since $\lambda\, (d\lambda)^{n-1}$ is a volume form on $L$ and $\tilde{\Sigma}$ is a graph over $\{t=0\}$, $\Omega$ is a volume form on $\tilde{\Sigma}$. Using \eqref{cvf}, one may compute 
\begin{align*}
i_{\partial_{\theta} - X_H}\Omega &= \lambda\, (d\lambda)^{n-1} + d\theta \, i_{X_H}\left(\lambda\, (d\lambda)^{n-1}\right) \\
&= \lambda\, (d\lambda)^{n-1} + H\, d\theta\, (d\lambda)^{n-1} + (n-1)\, d\theta\, \lambda \, dH\, (d\lambda)^{n-2}.
\end{align*}
On the other hand, let $\beta = \alpha\mid_{\tilde{\Sigma}}$. Then $\beta = H\, d\theta + \lambda$ and $d\beta = dH\, d\theta + d\lambda$ so that 
$$
\beta\, (d\beta)^{n-1} = \lambda\, (d\lambda)^{n-1} + H\, d\theta\, (d\lambda)^{n-1} + (n-1)\, d\theta\, \lambda \, dH\, (d\lambda)^{n-2} = i_{\partial_{\theta}-X_H}\Omega.
$$
By Lemma \ref{cflemma}, $\partial_{\theta} - X_H$ directs the characteristic foliation. 
\end{proof}

This lemma becomes useful in the context of Theorem \ref{maintheorem} when $X_H$ is a pseudo-gradient for a Morse function on $L$. In this case, $\tilde{\Sigma} \cong S^1\times L$ has Morse-Smale characteristic foliation. Indeed, there are finitely many hyperbolic periodic orbits directed by $\partial_{\theta}$ corresponding to the zeroes of $X_H$. 

The existence of a Morse function admitting a gradient-like contact vector field is the defining feature of a \textit{convex contact structure}, first introduced by Eliashberg and Gromov in \cite{EG} and studied further by Giroux in \cite{Giroux}. 

\begin{theorem}[Giroux, see \cite{CM, Sackel}]\label{girouxthm}
Every contact manifold admits a contact vector field which is gradient-like for some Morse function.
\end{theorem}

With this fact and Theorem \ref{maintheorem}, we have the following corollary. 

\begin{corollary}\label{cor}
Let $\Sigma^{2n}\subset (M^{2n+1}, \xi = \ker \alpha)$ be a hypersurface in a contact manifold diffeomorphic to $S^1\times L^{2n-1}$ for some closed manifold $L$. Suppose that the characteristic foliation $\Sigma_{\xi}$ consists of completely degenerate periodic orbits, so that the foliation is directed by $\partial_{\theta}$ for some choice of coordinate $\theta$ on $S^1$. Then there is an arbitrarily $C^{\infty}$-small perturbation of $\Sigma$ to a convex hypersurface. 
\end{corollary}

\begin{proof}
Because $\{\theta\}\times L \subset \Sigma$ is transverse to the characteristic foliation, $\alpha\mid_{\{\theta\}\times L}$ is contact. By Lemma \ref{conformallemma}, we may take a sufficiently small neighborhood of $\Sigma$ to be contactomorphic to $\R_t \times S^1_{\theta}\times L^{2n-1}$ with contact form $\alpha = t\, d\theta + \lambda$, where $\lambda$ is a contact form on $L$ and $\Sigma = \{t = 0\}$. By Theorem \ref{girouxthm}, we may choose a contact vector field $X_H$ on $L$ which is gradient-like for some Morse function on $K$. By scaling the corresponding contact Hamiltonian $H$, we may assume that the $C^{\infty}$ norm of $H$ is as small as we like. Then $\tilde{\Sigma} = \{t=H\}$ will be $C^{\infty}$-close to $\Sigma$, and by Lemma \ref{compute_perturb} the characteristic foliation of $\tilde{\Sigma}$ is directed by $\partial_{\theta} - X_H$. Since this vector field is Morse-Smale, by Theorem \ref{maintheorem}, $\Sigma$ is convex. 
\end{proof}

In particular, the proof of this corollary shows that any completely degenerate periodic orbit in a characteristic foliation can be locally perturbed to be hyperbolic.

\subsection{Mori's hypersurface}\label{morisubsection}
In \cite{Mori}, Mori introduced a particular non-convex hypersurface. We review the definition here. Consider 
$$
\R^{2n+1} = \R_{z} \times \R^2_{r,\theta} \times \R^2_{\rho_1, \phi_1} \times \cdots \times \R^2_{\rho_{n-1}, \phi_{n-1}}
$$
where $(r,\theta)$ and $(\rho_i, \phi_i)$ are polar coordinates in their respective planes. Let 
\begin{equation}\label{contactform}
\alpha = (2r^2 - 1)\, dz + r^2(r^2-1)\, d\theta + \sum_{i=1}^{n-1} \rho_i^2\, d\phi_i.    
\end{equation}
One can check that $\alpha$ is a contact form. Next, for $0 < \ve << 1$, let 
$$
\Sigma_0 = \left\{r^2 + \ve^{-2}\left(z^2 + \sum_{i=1}^{n-1} \rho_i^2\right) = 1 + \ve\right\}.
$$
Note that $\Sigma_0$ is diffeomorphic to $S^{2n}$. 

\begin{lemma}\cite{Mori}\label{foliationlemma}
The characteristic foliation on $\Sigma_0 \subset (\R^{2n-1}, \ker\alpha)$ is directed by the vector field 
\begin{multline}
X = \left[(r^2 - 1)^2 + (2r^2 - 1)(\ve^{-2}z^2 - \ve)\right]\, \partial_z \\
+ \ve^{-2}r(r^2-1)z\, \partial_r + (1 + 2\ve - 2\ve^{-2}z^2)\,\partial_{\theta} \\
+  \ve^{-2}(2r^2 - 1)z\, \sum_{i=1}^{n-1} \partial_{\rho_i} + \ve^{-2}(2r^4 - 2r^2 + 1)\, \sum_{i=1}^{n-1} \partial_{\phi_i}.
\end{multline}
\end{lemma}

We may visualize the characteristic foliation as follows \cite{Mori}. Observe that the vector field $X$ from Lemma \ref{foliationlemma} does not depend on $\theta$ or $\phi_i$. Thus, if we project $P:\Sigma_0\to \tilde{\Sigma}_0$ to the quarter ellipsoid $\tilde{\Sigma}_0 = \{z^2 + r^2 + \rho^2 = 1+\ve\}\subseteq \set{(z,r,\rho)}{r,\rho\geq 0}$ where $\rho^2 = \rho_1^1 + \cdots + \rho_{n-1}^2$, the vector field $X$ has a well-defined pushforward $\tilde{X}$ given by 
\begin{equation}
\tilde{X} = \left[(r^2 - 1)^2 + (2r^2 - 1)(\ve^{-2}z^2 - \ve)\right]\, \partial_z + \ve^{-2}r(r^2-1)z\, \partial_r + \sqrt{n-1}\ve^{-2}(2r^2 - 1)z\, \partial_{\rho}.    
\end{equation}
This pushforward is visualized in Figure \ref{fig:projections}. Observe that $\tilde{X}$ is Morse-Smale.

\begin{figure}[htb]
    \centering
    \subfloat[]{{\includegraphics[width=3.7cm]{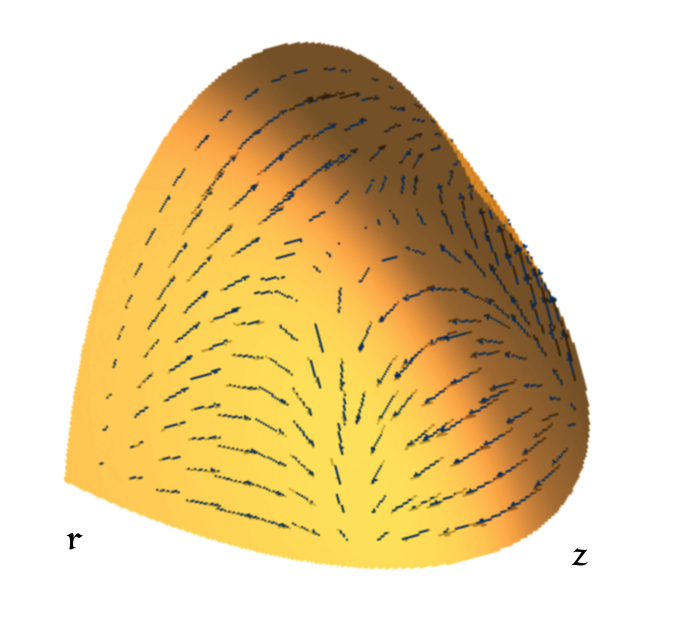} }}%
    \qquad
    \subfloat[]{{\includegraphics[width=5cm]{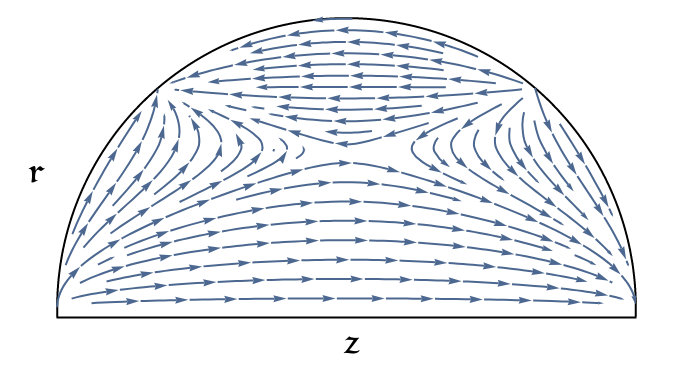} }}%
    \caption{(A) is the pushforward of $X$ to $\tilde{X}$ on the quarter ellipsoid $\tilde{\Sigma}$. (B) is the further projection of $\tilde{X}$ to the $(z,r)$-plane.}%
    \label{fig:projections}%
\end{figure}

In \cite{Mori} it was proven that $\Sigma_0$ is not convex. For completeness, we provide the argument here with some more details.  

\begin{lemma}\cite{Mori}\label{notconvex}
The hypersurface $\Sigma_0$ is not convex. 
\end{lemma}

\begin{proof}
Let $p = (0, r^*, 1+ \ve - (r^*)^2)$ denote the point on $\tilde{\Sigma}_0$ which is the hyperbolic zero of $\tilde{X}$. Observe that 
$$
P^{-1}(p) = \set{(\theta,\rho_1,\phi_1,\dots,\rho_{n-1},\phi_{n-1})}{\rho^2 = 1 + \ve - (r^*)^2}
$$
is diffeomorphic to $S^1_{\theta} \times S^{2n-3}$. The characteristic foliation along $P^{-1}(p)$ is directed by the vector field 
$$
(1 + 2\ve)\, \partial_{\theta} + \ve^{-2}(2(r^*)^4 - 2(r^*)^2 + 1)\, \sum_{i=1}^{n-1} \partial_{\phi_i}
$$
By adjusting $\ve$ if necessary, we may assume that this vector field foliates $S^1\times S^{2n-3}$ with periodic orbits, hence the characteristic foliation along $P^{-1}(p)$ consists of parallel leaves.

Suppose for the sake of contradiction that $\Sigma_0$ is convex. Then there is a dividing set $\Gamma$. Because $(\Sigma_0)_{\xi}$ is independent of $\theta$ and $\phi_i$, we may isotope $\Gamma$ so that $\Gamma = P^{-1}(C)$ for some multicurve $C \subset \tilde{\Sigma}_0$; see Figure \ref{fig:dividingsets}.

\begin{figure}[htb]
    \centering
    \subfloat[]{{\includegraphics[width=5cm]{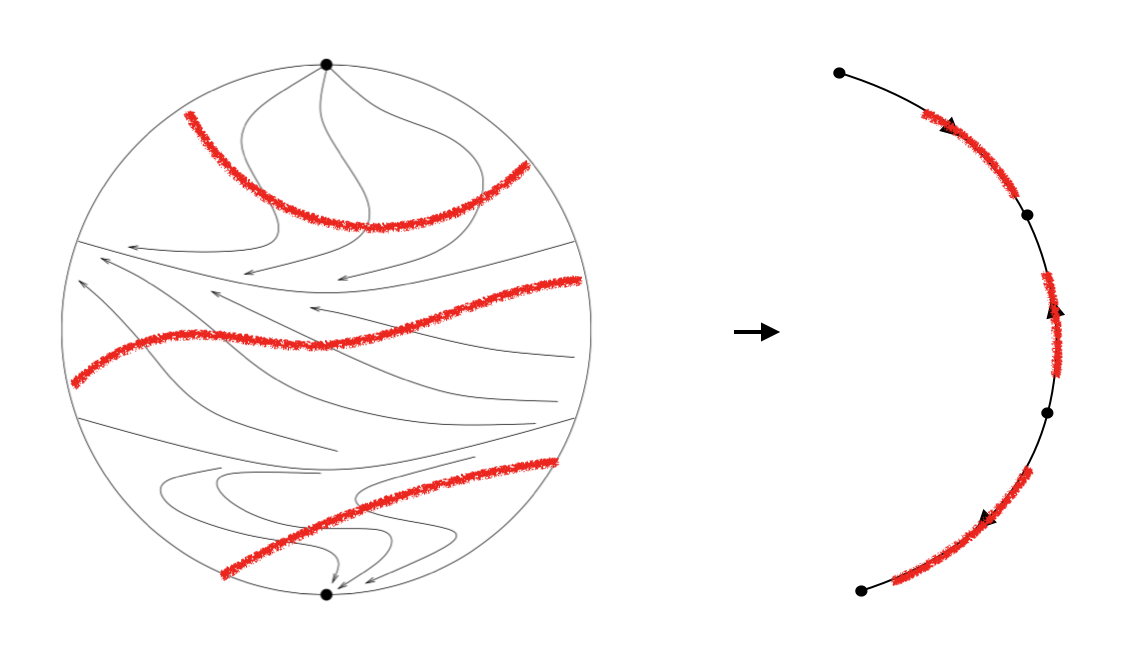} }}%
    \qquad
    \subfloat[]{{\includegraphics[width=5cm]{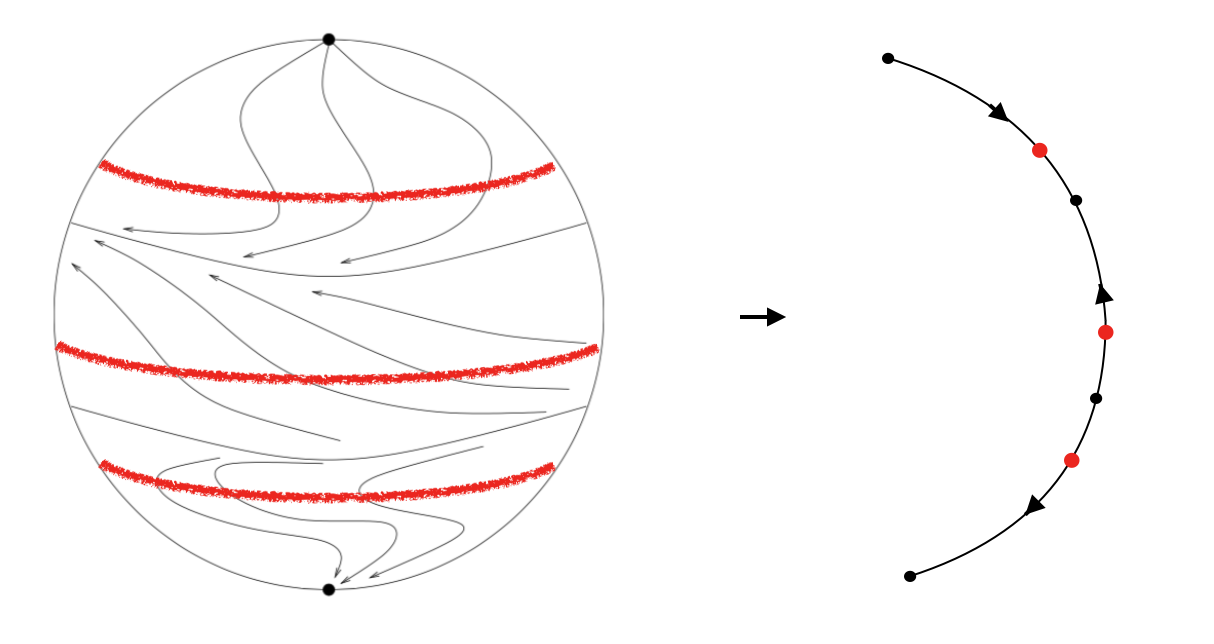} }}%
    \caption{(A) and (B) depict a sphere $\{z^2 + r^2 = c^2\}$ in $\R^3_{(r,\theta,z)}$ with a radially invariant characteristic foliation, together with the well-defined pushforward to the semicircle $\{z^2 + r^2 = c^2, r\geq 0\}$ in the $(r,z)$-plane. In (A), the dividing set in red projects to a codimension $0$ set on the semicircle. In (B), we choose a radially invariant dividing set which projects to a codimension $1$ submanifold of the semicircle. This is the low-dimensional analogue to the pushforward of $X$ to $\tilde{X}$.}%
    \label{fig:dividingsets}%
\end{figure}

We claim that $C$ does not contain $p$. Suppose it did: then $\Gamma$ contains the linearly foliated $P^{-1}(p)$. By \cite{Giroux}, there is a function $u:\Sigma \to \R$ for which $X(u) < 0$ on $P^{-1}(p)$, which contradicts the fact that $X$ has closed orbits on $P^{-1}(p)$. Thus, $C$ avoids $p$. Finally, note that the singular points of $X$ are
$$
(z = \pm \ve\sqrt{1+\ve}, r = 0, \theta, \rho_1 = 0, \phi_1, \dots, \rho_{n-1} = 0, \phi_{n-1}).
$$
For divergence reasons, these must lie in the negative and positive region, respectively. The remaining singular points of $\tilde{X}$ are 
$$
(z = \pm \ve\sqrt{\ve}, r = 1, \rho = 0)
$$
which lift under $P$ to periodic orbits that must lie in the positive and negative region, respectively. Consequently, $C$ must contain a component which is isotopic to one of the green curves in Figure \ref{fig:curves}.

\begin{figure}[htb]
    \centering
    \includegraphics[scale=0.45]{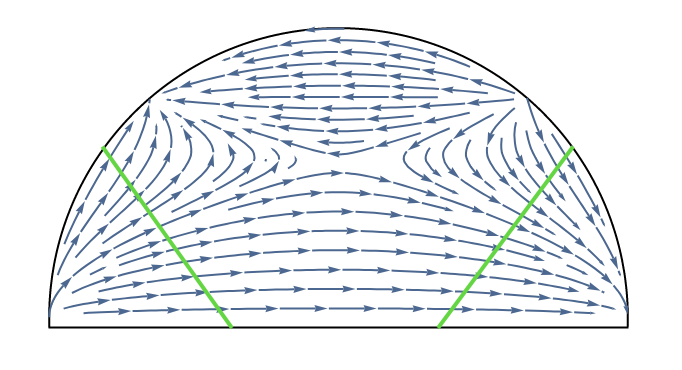}
    \caption{}
    \label{fig:curves}
\end{figure}

The lift of either of these curves under $P$ is diffeomorphic to $S^{2n-1}$. Moreover, there are necessarily other components of $\Gamma$. This contradicts a theorem of McDuff \cite{McDuff}, as the positive region of $\Sigma_0$ is then a symplectic manifold with convex boundary of the type $S^{2n-1} \sqcup S'$ for some other $2n-1$-manifold $S'$. Thus, no such dividing set can exist and so $\Sigma_0$ is not convex. 

\end{proof}

Using ideas inspired by the the proof of Corollary \ref{cor}, we prove Corollary \ref{maincorollary}.

\begin{proof}[Proof of Corollary \ref{maincorollary}.]
By Theorem \ref{maintheorem}, it suffices to perturb $\Sigma_0$ so that the resulting characteristic foliation is Morse-Smale. Lemma \ref{foliationlemma}, the subsequent discussion, and the proof of Lemma \ref{notconvex} show that the characteristic foliation is close to being Morse-Smale. The obstruction is $P^{-1}(p) \cong S^1\times S^{2n-3}$, which is foliated by parallel leaves. The pushforward $\tilde{X}$ (in the $\tilde{\Sigma}_0$ direction) is Morse-Smale, so it suffices to perturb the hypersurface near $P^{-1}(p)$ so that the resulting foliation, when restricted to $P^{-1}(p)$, is Morse-Smale.

Observe that for any fixed $\theta_0$, the contact form $\alpha$ in \eqref{contactform} restricts to the standard contact structure on $L = \{\theta_0\}\times S^{2n-3} \subseteq P^{-1}(p)$. Let $U$ be a small neighborhood of $p$ in $\tilde{\Sigma}_0$. Then $U\times L$ is transverse to the characteristic foliation and hence is also contact. Using the flow of the characteristic foliation starting at $U\times L$, we isolate a ``column'' $[0,1]_s \times U \times L$ where the characteristic foliation is directed by $\partial_s$. Note that we may take the foliation on top of the $L$ component to already be ``straight'', so this identification only straightens out the foliation above the $U$ component. By Lemma \ref{conformallemma}, we may assume that a neighborhood of the column is given by 
$$
\R_t \times [0,1]_s \times U \times L \quad \text{ with contact form } \quad t\, ds + \lambda
$$ 
where $\lambda$ is contact on $U\times L$, $\lambda\mid_L$ is the standard contact form on $S^{2n-3}$, and $\Sigma_0$ is identified with $\{t=0\}$. Finally, note that $[0,1]_s\times \{p\} \times L \subset P^{-1}(p)$. Our perturbation will be supported in this column.

Pick a $C^{\infty}$-small contact Hamiltonian $H:L\to \R$ such that the corresponding contact vector field $X_H$ is gradient-like for a Morse function. In particular, we may choose $X_H$ to be gradient-like for a height function on the sphere (\cite{EG}, \cite{Giroux}, \cite{Sackel}) so that $X_H$ has one source singularity and one sink singularity. Extend $H$ to $U\times L$ via a bump function which is constant near $p$. Finally, extend $H$ in the $s$ direction so that it is supported in the column $[0,1]_s \times U \times L$. Let $\Sigma_1 = \{t = H\}$. Because $\tilde{X}$ is hyperbolic at $p$ and thus structurally stable \cite{Palis}, the location of the zero may shift slightly from $p$ to some other point $p_1$, when perturbed as above, but the hyperbolic dynamics in the $\tilde{\Sigma}_0$ direction persist if the perturbation is small enough. As in Corollary \ref{cor}, the degenerate dynamics of the characteristic foliation on $P^{-1}(p_1)$ will be perturbed by the gradient-like vector field $X_H$ in the $L$ direction. As a result, the characteristic foliation of $\Sigma_1$ is Morse-Smale, as desired. 

\end{proof}

\bibliographystyle{amsplain}

\begin{thebibliography}{10}

\bibitem{CE} Cieliebak, K., Eliashberg, Y. \textit{From Stein to Weinstein and Back: Symplectic Geometry of Affine Manifolds}. American Mathematical Society, Colloquium Publications 59, 2012. 

\bibitem{CM} Courte, S., Massot, P. \textit{Contactomorphism groups and Legendrian flexibility}. arXiv preprint \textit{arXiv:1803.07997}, 2018. \url{https://arxiv.org/abs/1803.07997}.


\bibitem{EG} Eliashberg, Y., Gromov, M \textit{Convex Symplectic Manifolds}. Proc. Symp. Pure Math 52: 135 - 162, 1991.


\bibitem{Geiges} Geiges, H. \textit{An Introduction to Contact Topology}. Cambridge University Press, 2009. 

\bibitem{Giroux} Giroux, E. \textit{Convexity in Contact Topology.} Commentarii Mathematici Helvetici, 1991. 

\bibitem{Honda1} Honda, K. \textit{On the classification of tight contact structures I}. Geom. Topol. 4: 309 - 368, 2000.

\bibitem{HondaHuang2} Honda, K., Huang, Y. \textit{Bypass attachments in higher-dimensional contact topology}. arXiv preprint \textit{arXiv:1803.09142}, 2018. \url{https://arxiv.org/abs/1803.09142}.

\bibitem{HondaHuang} Honda, K., Huang, Y. \textit{Convex hypersurface theory in contact topology}. arXiv preprint \textit{arXiv:1907.06025}, 2019. \url{http://arxiv.org/abs/1907.06025}.

\bibitem{McDuff} McDuff, D. \textit{Symplectic manifolds with contact type boundaries}. Invent. Math. 103(1): 651 - 671, 1991.

\bibitem{MS} McDuff, D., Salamon, D. \textit{Introduction to Symplectic Topology: Third Edition}, Oxford University Press, 2017. 

\bibitem {Mori} Mori, A. \textit{On the violation of Thurston-Bennequin inequality for a certain non-convex hypersurface}. arXiv preprint arXiv:1111.0383, 2011. \url{https://arxiv.org/pdf/1111.0383.pdf}.

\bibitem {Mori2} Mori, A. \textit{Reeb Foliations on $S^5$ and contact $5$-manifolds violating the Thurston-Bennequin inequality}. arXiv preprint arXiv:0906.3237, 2009. \url{https://arxiv.org/abs/0906.3237}.

\bibitem{Palis} Palis, J., de Melo, W. \textit{Geometric Theory of Dynamical Systems}. Springer, New York, NY, 1982. 


\bibitem{Sackel} Sackel, K. \textit{Getting a handle on contact manifolds}. arXiv preprint \textit{arXiv:1905.11965}, 2019. \url{https://arxiv.org/abs/1905.11965}.



\end{thebibliography}

\end{document}